\documentclass[10pt]{amsart}
\usepackage{amssymb,MnSymbol}
\usepackage{amsthm,amsmath}
\usepackage{cite}

\title{Associative and preassociative functions\\ (improved version)$^*$}\thanks{$^*$This paper is an unpublished improved/corrected version of the article ``Associative and preassociative functions'' published in \emph{Semigroup Forum} in 2014.}

\author{Jean-Luc Marichal}
\address{Mathematics Research Unit, FSTC, University of Luxembourg \\
6, rue Coudenhove-Kalergi, L-1359 Luxembourg, Luxembourg} \email{jean-luc.marichal[at]uni.lu }

\author{Bruno Teheux}
\address{Mathematics Research Unit, FSTC, University of Luxembourg \\
6, rue Coudenhove-Kalergi, L-1359 Luxembourg, Luxembourg} \email{bruno.teheux[at]uni.lu }

\date{September 11, 2014}

\begin{document}

\theoremstyle{plain}
\newtheorem{theorem}{Theorem}[section]% Supprimer [section] pour une numérotation linéaire
\newtheorem{lemma}[theorem]{Lemma}
\newtheorem{proposition}[theorem]{Proposition}
\newtheorem{corollary}[theorem]{Corollary}
\newtheorem{fact}[theorem]{Fact}
\newtheorem*{main}{Main Theorem}

\theoremstyle{definition}
\newtheorem{definition}[theorem]{Definition}
\newtheorem{example}[theorem]{Example}

\theoremstyle{remark}
\newtheorem*{conjecture}{Conjecture}
\newtheorem{remark}{Remark}
\newtheorem{claim}{Claim}

\newcommand{\N}{\mathbb{N}}
\newcommand{\Q}{\mathbb{Q}}
\newcommand{\R}{\mathbb{R}}

\newcommand{\ran}{\mathrm{ran}}
\newcommand{\dom}{\mathrm{dom}}
\newcommand{\id}{\mathrm{id}}
\newcommand{\med}{\mathrm{med}}

\newcommand{\bfu}{\mathbf{u}}
\newcommand{\bfv}{\mathbf{v}}
\newcommand{\bfw}{\mathbf{w}}
\newcommand{\bfx}{\mathbf{x}}
\newcommand{\bfy}{\mathbf{y}}
\newcommand{\bfz}{\mathbf{z}}

\begin{abstract}
We investigate the associativity property for functions of indefinite arities and introduce and discuss the more general property of preassociativity, a generalization of associativity which does not involve any composition of functions.
\end{abstract}



\maketitle

%---------------------------------------------------------------------------------------------- Section
\section{Introduction}

Let $X,Y$ be arbitrary nonempty sets. Throughout this paper, we regard vectors $\bfx$ in $X^n$ as $n$-strings over $X$.
The $0$-string or \emph{empty} string is denoted by $\varepsilon$ so that $X^0=\{\varepsilon\}$. We denote by $X^*$ the set of all strings over $X$, that is, $X^*=\bigcup_{n\geqslant 0}X^n$.
Moreover, we consider $X^*$ endowed with concatenation for which we adopt the juxtaposition notation. For instance, if $\bfx\in X^n$, $y\in X$, and $\bfz\in X^m$, then $\bfx y\bfz\in X^{n+1+m}$. Furthermore, for $\bfx\in X^m$,  we use the short-hand notation $\bfx^n=\bfx\cdots\bfx\in X^{n\times m}$. The \emph{length} $|\bfx|$ of a string $\bfx\in X^*$ is a nonnegative integer defined in the usual way: we have $|\bfx|=n$ if and only if $\bfx\in X^n$.

In the sequel, we will be interested both in functions of a given fixed arity (i.e., \emph{$n$-ary} functions $F\colon X^n\to Y$) as well as in functions of indefinite arities (i.e., \emph{variadic} functions $F\colon X^*\to Y$). The \emph{$n$-ary part} $F_n$ of a variadic function $F\colon X^*\to Y$ is the restriction of $F$ to $X^n$, i.e., $F_n=F|_{X^n}$. We denote by $F^{\flat}$ the restriction $F|_{X^*\setminus\{\varepsilon\}}$ of $F$ to the strings of positive lengths. We say that a variadic function $F\colon X^*\to Y$ is \emph{standard} if the equality $F(\bfx)=F(\varepsilon)$ holds only if $\bfx =\varepsilon$. Finally, a variadic function $F\colon X^*\to X\cup\{\varepsilon\}$ is called a \emph{variadic operation on $X$} (or an \emph{operation} for short), and we say that such an operation is \emph{$\varepsilon$-preserving standard} (or \emph{$\varepsilon$-standard} for short) if it is standard and satisfies $F(\varepsilon)=\varepsilon$.

In this paper we are first interested in the following associativity property for variadic operations (see \cite[p.~24]{Mar98}).

\begin{definition}[{\cite{Mar98}}]\label{de:assocN}
A variadic operation $F\colon X^*\to X\cup\{\varepsilon\}$ is said to be \emph{associative} if for every $\bfx\bfy\bfz\in X^*$ we have $F(\bfx\bfy\bfz)=F(\bfx F(\bfy)\bfz)$.
\end{definition}

Any associative operation $F \colon X^* \to X\cup\{\varepsilon\}$ clearly satisfies the condition $F(\varepsilon)=F(F(\varepsilon))$. From this observation it follows immediately that any associative standard operation $F \colon X^* \to X\cup\{\varepsilon\}$ is necessarily $\varepsilon$-standard.

Alternative definitions of associativity for $\varepsilon$-standard operations have been proposed by different authors; see \cite[p.~16]{CalKolKomMes02}, \cite{CouMar11}, \cite[p.~32]{GraMarMesPap09}, \cite[p.~216]{KleMesPap20}, and \cite[p.~24]{Mar98}. It was proved in \cite{CouMar11} that these definitions are equivalent to Definition~\ref{de:assocN}.

Thus defined, associativity expresses that the function value of any string does not change when replacing any of its substrings with its corresponding value. As an example, the $\varepsilon$-standard operation $F\colon\R^*\to\R\cup\{\varepsilon\}$ defined by $F_n(\bfx)=\sum_{i=1}^nx_i$ for every integer $n\geqslant 1$ is associative.

Associative $\varepsilon$-standard operations $F\colon X^*\to X\cup\{\varepsilon\}$ are closely related to associative binary operations $G\colon X^2\to X$, which are defined as the solutions of the functional equation
$$
G(G(xy)z) ~=~ G(xG(yz)),\qquad x,y,z\in X.
$$
In fact, we show (Corollary~\ref{cor:sa57df6s}) that a binary operation $G\colon X^2\to X$ is associative if and only if there exists an associative $\varepsilon$-standard operation $F\colon X^*\to X\cup\{\varepsilon\}$ such that $G=F_2$.

Based on a recent investigation of associativity (see \cite{CouMar11,CouMar12}), we show that an associative $\varepsilon$-standard operation $F\colon X^*\to X\cup\{\varepsilon\}$ is completely determined by $F_1$ and $F_2$. We also provide necessary and sufficient conditions on $F_1$ and $F_2$ for an $\varepsilon$-standard operation $F\colon X^*\to X\cup\{\varepsilon\}$ to be associative (Theorem~\ref{thm:sdf87fs}). These results are gathered in Section 3.

The main aim of this paper is to introduce and investigate the following generalization of associativity, called \emph{preassociativity}.

\begin{definition}\label{de:7ads5sa}
We say that a function $F\colon X^*\to Y$ is \emph{preassociative} if for every $\bfx\bfy\bfy'\bfz\in X^*$ we have
$$
F(\bfy) ~=~ F(\bfy')\quad\Rightarrow\quad F(\bfx\bfy\bfz) ~=~ F(\bfx\bfy'\bfz).
$$
\end{definition}

Thus, a function $F\colon X^*\to Y$ is preassociative if the function value of any string does not change when modifying any of its substring without changing its value. For instance, any $\varepsilon$-standard operation $F\colon\R^*\to\R\cup\{\varepsilon\}$ defined by $F_n(\bfx)=f(\sum_{i=1}^nx_i)$ for every integer $n\geqslant 1$, where $f\colon\R\to\R$ is a one-to-one function, is preassociative.

It is immediate to see that any associative $\varepsilon$-standard operation $F\colon X^*\to X\cup\{\varepsilon\}$ necessarily satisfies the equation $F_1\circ F^{\flat}=F^{\flat}$ (take $\bfx\bfz=\varepsilon$ in Definition~\ref{de:assocN}). Actually, we show (Proposition~\ref{prop:A-PA1}) that an $\varepsilon$-standard operation $F\colon X^*\to X\cup\{\varepsilon\}$ is associative if and only if it is preassociative and satisfies $F_1\circ F^{\flat}=F^{\flat}$.

It is noteworthy that, contrary to associativity, preassociativity does not involve any composition of functions and hence allows us to consider a codomain $Y$ that may differ from $X\cup\{\varepsilon\}$. For instance, the length function $F\colon X^*\to\R$, defined by $F(\bfx)=|\bfx|$, is standard and preassociative.

In this paper we mainly focus on those preassociative functions $F\colon X^*\to Y$ for which $F_1$ and $F^{\flat}$ have the same range. (For $\varepsilon$-standard operations, the latter condition is an immediate consequence of the condition $F_1\circ F^{\flat}=F^{\flat}$ and hence these preassociative functions include the associative $\varepsilon$-standard operations). We show that these functions are completely determined by their nullary, unary, and binary parts (Proposition~\ref{prop:wer5}) and we provide necessary and sufficient conditions on $F_1$ and $F_2$ for a standard function $F\colon X^*\to Y$ to be preassociative and such that $F^{\flat}$ has the same range as $F_1$ (Theorem~\ref{thm:sdf87fs2}). We also give a characterization of these functions as compositions of the form $F^{\flat}=f\circ H^{\flat}$, where $H\colon X^*\to X\cup\{\varepsilon\}$ is an associative $\varepsilon$-standard operation and $f\colon H(X^*\setminus\{\varepsilon\})\to Y$ is one-to-one (Theorem~\ref{thm:FactoriAWRI-BPA237111}). This is done in Section 4.

The terminology used throughout this paper is the following. We denote by $\N$ the set $\{1,2,3,\ldots\}$ of strictly positive integers. The domain and range of any function $f$ are denoted by $\dom(f)$ and $\ran(f)$, respectively. The identity operation on $X$ is the function $\id\colon X\to X$ defined by $\id(x) = x$. Finally, for functions $f_i\colon X^{n_i}\to Y$ ($i=1,\ldots,m$), the function $(f_1,\ldots,f_m)$, from $X^{n_1+\cdots +n_m}$ to $Y^m$, is defined by $(f_1,\ldots,f_m)(\bfx_1\cdots\bfx_m)= f_1(\bfx_1)\cdots f_m(\bfx_m)$, where $|\bfx_i|=n_i$ for $i=1,\ldots,m$. For instance, given functions $f\colon Y^2\to Y$, $g\colon X\to Y$, and $h\colon X\to Y$, the function $f\circ(g,h)$, from $X^2$ to $Y$, is defined by $(f\circ(g,h))(x_1,x_2)=f(g(x_1),h(x_2))$.

\begin{remark}
Algebraically, preassociativity of a function $F\colon X^*\to Y$ means that $\ker(F)$ is a congruence on the free monoid $X^*$ generated by $X$ (see Proposition~\ref{prop:4.1-78f6dg}). Also, the condition that $F$ is standard means that $\varepsilon/\ker(F)=\{\varepsilon\}$. Similarly, an $\varepsilon$-standard operation $F\colon X^*\to X\cup\{\varepsilon\}$ is associative if and only if $F$ is preassociative and $F(\bfx)\in\bfx/\ker(F)$ for every $\bfx\in X^*$ (see Proposition~\ref{prop:A-PA1}). Thus, associativity and preassociativity have good algebraic translations in terms of congruences of free monoids. It turns out that this algebraic language does not help either in stating or in proving the results that we obtain in this paper. However, this translation deserves further investigation and could lead to characterizations of certain classes of associative or preassociative functions. We postpone this investigation to a future paper.
\end{remark}

%---------------------------------------------------------------------------------------------- Section
\section{Preliminaries}

Recall that a function $F\colon X^n\to X$ ($n\in\N$) is said to be \emph{idempotent} (see, e.g., \cite{GraMarMesPap09}) if $F(x^n)=x$ for every $x\in X$.

We now introduce the following definitions. We say that an $\varepsilon$-standard operation $F\colon X^*\to X\cup\{\varepsilon\}$ is
\begin{itemize}
\item \emph{idempotent} if $F_n$ is idempotent for every $n\in\N$,

\item \emph{unarily idempotent} if $F_1=\id$,

\item \emph{unarily range-idempotent} if $F_1|_{\ran(F^{\flat})}=\id|_{\ran(F^{\flat})}$, or equivalently, $F_1\circ F^{\flat}=F^{\flat}$. In this case $F_1$ necessarily satisfies the equation $F_1\circ F_1=F_1$.
\end{itemize}
We say that a function $F\colon X^*\to Y$ is \emph{unarily quasi-range-idempotent} if $\ran(F_1)=\ran(F^{\flat})$. We observe that this property is a consequence of the condition $F_1\circ F^{\flat}=F^{\flat}$ whenever $F$ is an $\varepsilon$-standard operation. The following proposition provides a finer result.

\begin{proposition}\label{prop:22f1f1f1}
An $\varepsilon$-standard operation $F\colon X^*\to X\cup\{\varepsilon\}$ is unarily range-idempotent if and only if it is unarily quasi-range-idempotent and satisfies $F_1\circ F_1=F_1$.
\end{proposition}

\begin{proof}
(Necessity) We have $\ran(F_1)\subseteq\ran(F^{\flat})$ for any operation $F\colon X^*\to X\cup\{\varepsilon\}$. Since $F$ is unarily range-idempotent, we have $F_1\circ F^{\flat}=F^{\flat}$, from which the converse inclusion follows immediately. In particular, $F_1\circ F_1=F_1$.

(Sufficiency) Since $F$ is unarily quasi-range-idempotent, the identity $F_1\circ F_1=F_1$ is equivalent to $F_1\circ F^{\flat}=F^{\flat}$.
\end{proof}

We now show that any unarily quasi-range-idempotent function $F\colon X^*\to Y$ can always be factorized as $F^{\flat}=F_1\circ H^{\flat}$, where $H\colon X^*\to X\cup\{\varepsilon\}$ is a unarily range-idempotent $\varepsilon$-standard operation.

First recall that a function $g$ is a \emph{quasi-inverse} \cite[Sect.~2.1]{SchSkl83} of a function $f$ if
$$
f\circ g|_{\ran(f)}=\id|_{\ran(f)}\qquad\mbox{and}\qquad\ran(g|_{\ran(f)})=\ran(g).
$$

For any function $f$, denote by $Q(f)$ the set of its quasi-inverses. This set is nonempty whenever we assume the Axiom of Choice (AC), which is actually just another
form of the statement ``every function has a quasi-inverse.'' Recall also that the relation of being quasi-inverse is symmetric, i.e., if $g\in
Q(f)$, then $f\in Q(g)$; moreover, we have $\ran(f)\subseteq\dom(g)$, $\ran(g)\subseteq\dom(f)$, and the functions $f|_{\ran(g)}$ and $g|_{\ran(f)}$ are one-to-one.

\begin{proposition}\label{prop:QRIqi2431}
Assume AC and let $F\colon X^*\to Y$ be a unarily quasi-range-idempo{\-}tent function. For any $g\in Q(F_1)$, the $\varepsilon$-standard operation $H\colon X^*\to X\cup\{\varepsilon\}$ defined as $H^{\flat}=g\circ F^{\flat}$ is a unarily range-idempotent solution of the equation $F^{\flat}=F_1\circ H^{\flat}$. Moreover, the function $F_1|_{\ran(H^{\flat})}$ is one-to-one.
\end{proposition}

\begin{proof}
Since $\ran(F_1)=\ran(F^{\flat})$, we have $F_1\circ g|_{\ran(F^{\flat})}=\id|_{\ran(F^{\flat})}$ and hence $F_1\circ H^{\flat}=F_1\circ g\circ F^{\flat}=F^{\flat}$. Also, $H$ is unarily range-idempotent since $H_1\circ H^{\flat}=g\circ F_1\circ H^{\flat}=g\circ F^{\flat}=H^{\flat}$. Since $F_1|_{\ran(g)}$ is one-to-one and $\ran(H^{\flat})\subseteq\ran(g)$, the function $F_1|_{\ran(H^{\flat})}$ is one-to-one, too.
\end{proof}

The following proposition yields necessary and sufficient conditions for a function $F\colon X^*\to Y$ to be unarily quasi-range-idempotent. We first consider a lemma.

\begin{lemma}\label{lemma:sdaf6f}
Let $f$ and $g$ be functions. If there exists a function $h$ such that $f=g\circ h$, then $\ran(f)\subseteq\ran(g)$. Under AC, the converse also holds.
\end{lemma}

\begin{proof}
The necessity is trivial. For the sufficiency, take $e\in Q(g)$. Since $\ran(f)\subseteq\ran(g)$ we must have $g\circ e|_{\ran(f)}=\id|_{\ran(f)}$, or equivalently, $g\circ e\circ f=f$.
\end{proof}

\begin{proposition}\label{prop:ACqriTFAE3451}
Assume AC and let $F\colon X^*\to Y$ be a function. The following assertions are equivalent.
\begin{enumerate}
\item[(i)] $F$ is unarily quasi-range-idempotent.

\item[(ii)] There exists an $\varepsilon$-standard operation $H\colon X^*\to X\cup\{\varepsilon\}$ such that $F^{\flat}=F_1\circ H^{\flat}$.

\item[(iii)] There exists a unarily idempotent $\varepsilon$-standard operation $H\colon X^*\to X\cup\{\varepsilon\}$ and a function $f\colon X\to Y$ such that $F^{\flat}=f\circ H^{\flat}$. In this case, $f=F_1$.

\item[(iv)] There exists a unarily range-idempotent $\varepsilon$-standard operation $H\colon X^*\to X\cup\{\varepsilon\}$ and a function $f\colon X\to Y$ such that $F^{\flat}=f\circ H^{\flat}$. In this case, $F^{\flat}=F_1\circ H^{\flat}$. Moreover, if $h=F_1|_{\ran(H^{\flat})}$ is one-to-one, then $h^{-1}\in Q(F_1)$.

\item[(v)] There exists a unarily quasi-range-idempotent $\varepsilon$-standard operation $H\colon X^*\to X\cup\{\varepsilon\}$ and a function $f\colon X\to Y$ such that $F^{\flat}=f\circ H^{\flat}$.
\end{enumerate}
In assertions (ii), (iv), and (v) we may choose $H^{\flat}=g\circ F^{\flat}$ for any $g\in Q(F_1)$ and $H$ is then unarily range-idempotent. In assertion (iii) we may choose $H_1=\id$ and $H_n=g\circ F_n$ for every $n>1$ and any $g\in Q(F_1)$.
\end{proposition}

\begin{proof}
$(i)\Rightarrow (ii)$ Follows from Lemma~\ref{lemma:sdaf6f} or Proposition~\ref{prop:QRIqi2431}.

$(ii)\Rightarrow (iii)$ Modifying $H_1$ into $\id$ and taking $f=F_1$, we obtain $F^{\flat}=f\circ H^{\flat}$, where $H$ is unarily idempotent. We then have $F_1=f\circ H_1=f\circ\id =f$.

$(iii)\Rightarrow (iv)$ The first part is trivial. Also, we have $F_1\circ H^{\flat}=f\circ H_1\circ H^{\flat}=f\circ H^{\flat}=F^{\flat}$. Now, if $h=F_1|_{\ran(H^{\flat})}$ is one-to-one, then we have $H^{\flat}=h^{-1}\circ F^{\flat}$ and hence $F_1\circ h^{-1}\circ F_1=F_1\circ H_1=h\circ H_1\circ H_1=h\circ H_1=F_1$, which shows that $h^{-1}\in Q(F_1)$.

$(iv)\Rightarrow (v)$ Trivial.

$(v)\Rightarrow (i)$ We have $\ran(F_1)=\ran(f\circ H_1)=\ran(f\circ H^{\flat})=\ran(F^{\flat})$.

The last part follows from Proposition~\ref{prop:QRIqi2431}.
\end{proof}

It is noteworthy that the proof of implication $(v)\Rightarrow (i)$ in Proposition~\ref{prop:ACqriTFAE3451} shows that the property of unary quasi-range-idempotence is preserved under left composition with unary maps.

%---------------------------------------------------------------------------------------------- Section
\section{Associative functions}

The following proposition shows that the definition of associativity (Definition~\ref{de:assocN}) remains unchanged if we upper bound the length of the string $\bfx\bfz$ by one. The proof makes use of the preassociativity property and will be postponed to Section~\ref{sec:PA}.

\begin{proposition}\label{prop:AsimpEquivVersion}
An $\varepsilon$-standard operation $F\colon X^*\to X\cup\{\varepsilon\}$ is associative if and only if for every $\bfx\bfy\bfz\in X^*$ such that $|\bfx\bfz|\leqslant 1$ we have $F(\bfx\bfy\bfz)=F(\bfx F(\bfy)\bfz)$.
\end{proposition}

As observed in \cite[p.~25]{Mar98} (see also \cite[p.~15]{BelPraCal07} and \cite[p.~33]{GraMarMesPap09}), associative $\varepsilon$-standard operations $F\colon X^*\to X\cup\{\varepsilon\}$ are completely determined by their unary and binary parts. Indeed, by associativity we have
\begin{equation}\label{eq:saf76sf5xx}
F_n(x_1\cdots x_n) ~=~ F_2(F_{n-1}(x_1\cdots x_{n-1})x_n),\qquad n\geqslant 3,
\end{equation}
or equivalently,
\begin{equation}\label{eq:saf76sf5}
F_n(x_1\cdots x_n) ~=~ F_2(F_2(\cdots F_2(F_2(x_1x_2)x_3)\cdots) x_n),\qquad n\geqslant 3.
\end{equation}
We state this immediate result as follows.

\begin{proposition}\label{prop:AssocUniquen45}
Let $F\colon X^*\to X\cup\{\varepsilon\}$ and $G\colon X^*\to X\cup\{\varepsilon\}$ be two associative $\varepsilon$-standard operations such that $F_1=G_1$ and $F_2=G_2$. Then $F=G$.
\end{proposition}

A natural and important question now arises: Find necessary and sufficient conditions on $F_1$ and $F_2$ for an $\varepsilon$-standard operation $F\colon X^*\to X\cup\{\varepsilon\}$ to be associative. The following proposition is an important step towards an answer to this question.

\begin{proposition}\label{thm:Acz12new}
An $\varepsilon$-standard operation $F\colon X^*\to X\cup\{\varepsilon\}$ is associative if and only if
\begin{enumerate}
\item[(i)] $F_1\circ F_1=F_1$ and $F_1\circ F_2=F_2$,

\item[(ii)] $F_2=F_2\circ (F_1,\id)=F_2\circ (\id,F_1)$,

\item[(iii)] $F_2$ is associative, and

\item[(iv)] condition (\ref{eq:saf76sf5xx}) or (\ref{eq:saf76sf5}) holds.
\end{enumerate}
\end{proposition}

\begin{proof}
The necessity is trivial. To prove the sufficiency, let $F\colon X^*\to X\cup\{\varepsilon\}$ be an $\varepsilon$-standard operation satisfying conditions (i)--(iv) and let us show that $F$ is associative. Using conditions (ii)--(iv) we show by induction on $n$ that
$$
F_n ~=~ F_2\circ (F_{n-1},\id) ~=~ F_2\circ (\id,F_{n-1}) \qquad\mbox{for every $n\geqslant 2$.}
$$
The result clearly holds for $n=2$ by (ii). Assume that it holds for some $n\geqslant 2$ and let $u\bfx vw\in X^{n+1}$. We have
\begin{eqnarray*}
F_{n+1}(u\bfx vw) &=& F_2(F_n(u\bfx v)w) ~=~ F_2(F_2(F_{n-1}(u\bfx)v)w)\\
&=& F_2(F_2(uF_{n-1}(\bfx v))w) ~=~ F_2(uF_2(F_{n-1}(\bfx v)w))\\
&=& F_2(uF_n(\bfx vw)),
\end{eqnarray*}
which completes the proof by induction.

Thus, we have proved that $F(x\bfy)=F(xF(\bfy))$ and $F(\bfy z)=F(F(\bfy)z)$ for every $x\bfy z\in X^*$. To see that $F$ is associative, by Proposition~\ref{prop:AsimpEquivVersion} it remains to show that $F(\bfy)=F(F(\bfy))$ for every $\bfy\in X^*$. By (i) we can assume that $|\bfy|\geqslant 3$. Setting $\bfy=\bfu y$, by (i) we have
$$
F(\bfy) ~=~ F(\bfu y) ~=~ F_2(F(\bfu)y) ~=~ F_1(F_2(F(\bfu)y)) ~=~ F_1(F(\bfu y)) ~=~ F(F(\bfy)).
$$
This completes the proof.
\end{proof}

\begin{corollary}\label{cor:sa57df6s}
A binary operation $F\colon X^2\to X$ is associative if and only if there exists an associative $\varepsilon$-standard operation $G\colon X^*\to X\cup\{\varepsilon\}$ such that $F=G_2$.
\end{corollary}

\begin{proof}
The sufficiency follows from Proposition~\ref{thm:Acz12new}. For the necessity, just take $G_1=\id$.
\end{proof}

The following theorem, which follows from Proposition~\ref{thm:Acz12new}, provides an answer to the question raised above.

\begin{theorem}\label{thm:sdf87fs}
Let $F_1\colon X\to X$ and $F_2\colon X^2\to X$ be two operations. Then there exists an associative $\varepsilon$-standard operation $G\colon X^*\to X\cup\{\varepsilon\}$ such that $G_1=F_1$ and $G_2=F_2$ if and only if conditions (i)--(iii) of Proposition~\ref{thm:Acz12new} hold. Such an operation $G$ is then uniquely determined by $G_n=G_2\circ (G_{n-1},\id)$ for $n\geqslant 3$.
\end{theorem}

Thus, two operations $F_1\colon X\to X$ and $F_2\colon X^2\to X$ are the unary and binary parts of an associative $\varepsilon$-standard operation $F\colon X^*\to X\cup\{\varepsilon\}$ if and only if these operations satisfy conditions (i)--(iii) of Proposition~\ref{thm:Acz12new}. In the case when only a binary operation $F_2$ is given, any unary operation $F_1$ satisfying conditions (i) and (ii) can be considered, for instance the identity operation. Note that it may happen that the identity operation is the sole possibility for $F_1$, for instance when we consider the binary operation $F_2\colon\R^2\to\R$ defined by $F_2(x_1x_2)=x_1+x_2$. However, there are examples where $F_1$ may differ from the identity operation. For instance, for any real number $p\geqslant 1$, the $\varepsilon$-standard \emph{$p$-norm} operation $F\colon\R^*\to\R\cup\{\varepsilon\}$ defined by $F_n(\bfx) = (\sum_{i=1}^n|x_i|^p)^{1/p}$ for every $n\in\N$, is associative but not unarily idempotent (here $|x|$ denotes the absolute value of $x$). Of course an associative $\varepsilon$-standard operation $F\colon X^*\to X\cup\{\varepsilon\}$ that is not unarily idempotent can be made unarily idempotent simply by setting $F_1=\id$. By Proposition~\ref{thm:Acz12new} the resulting operation is still associative.

%---------------------------------------------------------------------------------------------- Section
\section{Preassociative functions}
\label{sec:PA}

In this section we investigate the preassociativity property (see Definition~\ref{de:7ads5sa}) and characterize certain subclasses of preassociative functions (Theorem~\ref{thm:FactoriAWRI-BPA237111} and Proposition~\ref{prop:sdf57dzzs}).

Just as for associativity, preassociativity may have different equivalent forms. The following straightforward proposition gives an equivalent definition based on two equalities of values. The extension to any number of equalities is immediate.

\begin{proposition}\label{prop:4.1-78f6dg}
A function $F\colon X^*\to Y$ is preassociative if and only if for every $\bfx\bfx'\bfy\bfy'\in X^*$ we have
$$
F(\bfx) ~=~ F(\bfx')\quad\mbox{and}\quad F(\bfy) ~=~ F(\bfy')\quad\Rightarrow\quad F(\bfx\bfy) ~=~ F(\bfx'\bfy').
$$
\end{proposition}

The following immediate result provides a simplified but equivalent definition of preassociativity (exactly as Proposition~\ref{prop:AsimpEquivVersion} did for associativity).

\begin{proposition}\label{prop:AsimpEquivVersionP}
A function $F\colon X^*\to Y$ is preassociative if and only if for every $\bfx\bfy\bfy'\bfz\in X^*$ such that $|\bfx\bfz|= 1$ we have
$$
F(\bfy) ~=~ F(\bfy')\quad\Rightarrow\quad F(\bfx\bfy\bfz) ~=~ F(\bfx\bfy'\bfz).
$$
\end{proposition}

As mentioned in the introduction, any associative operation $F\colon X^*\to X\cup\{\varepsilon\}$ is preassociative. Moreover, we have the following result.

\begin{proposition}\label{prop:A-PA1}
An $\varepsilon$-standard operation $F\colon X^*\to X\cup\{\varepsilon\}$ is associative if and only if it is preassociative and unarily range-idempotent (i.e., $F_1\circ F^{\flat}=F^{\flat}$).
\end{proposition}

\begin{proof}
(Necessity) $F$ is clearly unarily range-idempotent. To see that it is preassociative, let $\bfx\bfy\bfy'\bfz\in X^*$ such that $F(\bfy)=F(\bfy')$. Then we have $F(\bfx\bfy\bfz)=F(\bfx F(\bfy)\bfz)=F(\bfx F(\bfy')\bfz)=F(\bfx\bfy'\bfz)$.

(Sufficiency)  Let $\bfx\bfy\bfz\in X^*$. We then have $F(\bfy)=F(F(\bfy))$ and hence $F(\bfx\bfy\bfz)=F(\bfx F(\bfy)\bfz)$.
\end{proof}

\begin{remark}
\begin{enumerate}
\item[(a)] From Proposition~\ref{prop:A-PA1} it follows that a preassociative and unarily idempotent (i.e., $F_1=\id$) $\varepsilon$-standard operation $F\colon X^*\to X\cup\{\varepsilon\}$ is necessarily associative.

\item[(b)] The $\varepsilon$-standard operation $F\colon\R^*\to\R\cup\{\varepsilon\}$ defined by $F_n(\bfx)=2\sum_{i=1}^nx_i$ for every $n\in\N$ is an instance of a preassociative operation which is not associative.
\end{enumerate}
\end{remark}

We are now ready to provide a very simple proof of Proposition~\ref{prop:AsimpEquivVersion}.

\begin{proof}[Proof of Proposition~\ref{prop:AsimpEquivVersion}]
The necessity is trivial. To prove the sufficiency, let $F\colon X^*\to X\cup\{\varepsilon\}$ satisfy the stated conditions. Then $F$ is clearly unarily range-idempotent. To see that it is associative, by Proposition~\ref{prop:A-PA1} it suffices to show that it is preassociative. Let $\bfx\bfy\bfy'\bfz\in X^*$ such that $|\bfx\bfz|= 1$ and assume that $F(\bfy)=F(\bfy')$. Then we have $F(\bfx\bfy\bfz)=F(\bfx F(\bfy)\bfz)=F(\bfx F(\bfy')\bfz)=F(\bfx\bfy'\bfz)$. The conclusion then follows from Proposition~\ref{prop:AsimpEquivVersionP}.
\end{proof}

The following two straightforward propositions show how new preassociative functions can be generated from given preassociative standard functions by compositions with unary maps.

\begin{proposition}[Right composition]
If $F\colon X^*\to Y$ is standard and preassociative then, for every function $g\colon X'\to X$, the function $H\colon X'^*\to Y$, defined as $H_0=\mathbf{a}$ for some $\mathbf{a}\in Y\setminus\ran(F^{\flat})$ and $H_n=F_n\circ(g,\ldots,g)$ for every $n\in\N$, is standard and preassociative. For instance, the $\varepsilon$-standard squared distance operation $F\colon \R^*\to\R\cup\{\varepsilon\}$ defined as $F_n(\bfx)=\sum_{i=1}^n x_i^2$ for every $n\in\N$ is preassociative.
\end{proposition}

\begin{proposition}[Left composition]\label{prop:leftcomp56}
Let $F\colon X^*\to Y$ be a preassociative standard function and let $g\colon Y\to Y'$ be a function. If $g|_{\ran(F^{\flat})}$ is one-to-one, then the function $H\colon X^*\to Y$ defined as $H_0=\mathbf{a}$ for some $\mathbf{a}\in Y'\setminus\ran(g|_{\ran(F^{\flat})})$ and $H^{\flat}=g\circ F^{\flat}$ is standard and preassociative. For instance, the $\varepsilon$-standard operation $F\colon \R^*\to\R\cup\{\varepsilon\}$ defined as $F_n(\bfx)=\exp(\sum_{i=1}^nx_i)$ for every $n\in\N$ is preassociative.
\end{proposition}

\begin{remark}\label{rem:s8d76}
\begin{enumerate}
\item[(a)] If $F\colon X^*\to Y$ is a preassociative standard function and $(g_n)_{n\in\N}$ is a sequence of functions from $X'$ to $X$, then the function $H\colon X'^*\to Y$, defined by $H_0=\mathbf{a}$ for some $\mathbf{a}\in Y\setminus\ran(F^{\flat})$ and $H_n=F_n\circ(g_n,\ldots,g_n)$ for every $n\in\N$, need not be preassociative. For instance, consider the $\varepsilon$-standard sum operation $F_n(\bfx)=\sum_{i=1}^nx_i$ over the reals and the sequence $g_n(x)=\exp(nx)$. Then, for $x_1=\log(1)$, $x_2=\log(2)$, $x'_1=\frac{1}{2}\log(3)$, $x'_2=\frac{1}{2}\log(2)$, and $x_3=0$, we have $H(x_1x_2)=H(x'_1x'_2)$ but $H(x_1x_2x_3)\neq H(x'_1x'_2x_3)$.

\item[(b)] Preassociativity is not always preserved by left composition of a preassociative function with a unary map. For instance, consider the $\varepsilon$-standard sum operation $F_n(\bfx)=\sum_{i=1}^nx_i$ over the reals and let $g(x)=\max\{x,0\}$. Then for the $\varepsilon$-standard operation $H^{\flat}=g\circ F^{\flat}$, we have $H(-1,-2)=0=H(-1,1)$ but $H(-1,-2,1)=0\neq 1=H(-1,1,1)$. Thus $H$ is not preassociative.
\end{enumerate}
\end{remark}

Although preassociativity generalizes associativity, it remains a rather strong property, especially when the functions have constant $n$-ary parts. The following result illustrates this observation.

\begin{proposition}
Let $F\colon X^*\to Y$ be a preassociative function.
\begin{enumerate}
\item[(a)] If $F_n$ is constant for some $n\in\N$, then so is $F_{n+1}$.

\item[(b)] If $F_n$ and $F_{n+1}$ are the same constant function $c$, then $F_m=c$ for all $m\geqslant n$.
\end{enumerate}
\end{proposition}

\begin{proof}
(a) For every $\bfy,\bfy'\in X^n$ and every $x\in X$, we have $F(\bfy)=F(\bfy')=c_n$ and hence $F(x\bfy)=F(x\bfy')$. This means that $F_{n+1}$ depends only on its first argument. Similarly we show that it depends only on its last argument.

(b) Let $\bfx yz\in X^{n+2}$. Then $c=F(\bfx)=F(\bfx y)$ and hence $c=F(\bfx z)=F(\bfx yz)$. So $F_{n+2}=c$, etc.
\end{proof}

We now focus on those preassociative functions $F\colon X^*\to Y$ which are unarily quasi-range-idempotent, that is, such that $\ran(F_1)=\ran(F^{\flat})$. As we will now show, this special class of functions has very interesting and even surprising properties. First of all, preassociative and unarily quasi-range-idempotent functions are completely determined by their nullary, unary, and binary parts.

\begin{proposition}\label{prop:wer5}
Let $F\colon X^*\to Y$ and $G\colon X^*\to Y$ be preassociative and unarily quasi-range-idempotent functions such that $F_0=G_0$, $F_1=G_1$, and $F_2=G_2$, then $F=G$.
\end{proposition}

\begin{proof}
Let $F\colon X^*\to Y$ and $G\colon X^*\to Y$ be two function satisfying the stated conditions. We show by induction on $n\in\N$ that $F_n=G_n$. The result clearly holds for $n\leqslant 2$. Suppose that it holds for $n-1\geqslant 1$ and show that it still holds for $n$. Let $\bfx\in X^n$ and choose $z\in X$ such that $F(z)=F(x_1\cdots x_{n-1})$. By induction hypothesis, we have $G(z)=G(x_1\cdots x_{n-1})$. Therefore by preassociativity we have $F_n(\bfx)=F_2(zx_n)=G_2(zx_n)=G_n(\bfx)$.
\end{proof}

\begin{remark}\label{rem:s8d76xd}
Proposition~\ref{prop:wer5} states that any element of the class of preassociative and unarily quasi-range-idempotent functions is completely determined inside this class by its nullary, unary, and binary parts. This property is shared by other classes of preassociative functions. Consider for example the class of standard functions $F\colon X^*\rightarrow Y$ for which there are distinct $c,c'\in Y$ such that $F_1=c$ and $F_n=c'$ for all $n\geqslant 2$. The elements of this class are preassociative functions that are not unarily quasi-range-idempotent. However, any function of this class is completely determined inside the class by its nullary, unary, and binary parts.
\end{remark}

We now give a characterization of the preassociative and unarily quasi-range-idem\-po\-tent standard functions as compositions of associative $\varepsilon$-standard operations with one-to-one unary maps. We first consider a lemma, which provides equivalent conditions for a unarily quasi-range-idempotent function to be preassociative.

\begin{lemma}\label{lemma:UQRIrew67}
Assume AC and let $F\colon X^*\to Y$ be a unarily quasi-range-idempotent function. Consider the following assertions.
\begin{enumerate}
\item[(i)] $F$ is preassociative.

\item[(ii)] For every $g\in Q(F_1)$, the $\varepsilon$-standard operation $H\colon X^*\to X\cup\{\varepsilon\}$ defined by $H^{\flat}=g\circ F^{\flat}$ is associative.

\item[(iii)] There is $g\in Q(F_1)$ such that the $\varepsilon$-standard operation $H\colon X^*\to X\cup\{\varepsilon\}$ defined by $H^{\flat}=g\circ F^{\flat}$ is associative.
\end{enumerate}
Then $(i) \Rightarrow (ii) \Leftrightarrow (iii)$. If $F$ is standard, then $(iii) \Rightarrow (i)$.
\end{lemma}

\begin{proof}
$(i)\Rightarrow (ii)$  By Proposition~\ref{prop:QRIqi2431}, $H$ is unarily range-idempotent. Since $g|_{\ran(F_1)}=g|_{\ran(F^{\flat})}$ is one-to-one, we have that $H$ is preassociative. It follows that $H$ is associative by Proposition~\ref{prop:A-PA1}.

$(ii)\Rightarrow (iii)$ Trivial.

$(iii)\Rightarrow (ii)$ Let $g$ and $H$ be as defined in (iii). Let $g'\in Q(F_1)$ and consider the $\varepsilon$-standard operation $H'\colon X^*\to X\cup\{\varepsilon\}$ defined by $H'^{\flat}=g'\circ F^{\flat}$. By Proposition~\ref{prop:QRIqi2431}, $H'$ is unarily range-idempotent. Since we have $H'^{\flat}=g'\circ F_1\circ g\circ F^{\flat}=g'\circ F_1\circ H^{\flat}$ and the functions $F_1|_{\ran(H^{\flat})}$ and $g'|_{\ran(F_1)}$ are one-to-one, we have that $H'$ is preassociative. It follows that $H'$ is associative by Proposition~\ref{prop:A-PA1}.

$(iii)\Rightarrow (i)$ By Proposition~\ref{prop:A-PA1} we have that $H$ is preassociative. Since $g|_{\ran(F^{\flat})}$ is a one-to-one function from $\ran(F^{\flat})$ onto $\ran(H^{\flat})$, we have $F^{\flat}=(g|_{\ran(F^{\flat})})^{-1}\circ H^{\flat}$. By Proposition~\ref{prop:leftcomp56} it follows that $F$ is preassociative.
\end{proof}

\begin{remark}
Let $F\colon X^*\to Y$ be a preassociative standard function satisfying $F^{\flat}=f\circ H^{\flat}$, where $f\colon X\to Y$ is any function and $H\colon X^*\to X\cup\{\varepsilon\}$ is any unarily range-idempotent $\varepsilon$-standard operation. Then $F^{\flat}=F_1\circ H^{\flat}$ by Proposition~\ref{prop:ACqriTFAE3451}(iv). However, $H$ need not be associative. For instance, if $F^{\flat}$ is a constant function, then $H$ could be any unarily range-idempotent $\varepsilon$-standard operation. However, Lemma~\ref{lemma:UQRIrew67} shows that, assuming AC, there is always an associative $\varepsilon$-standard operation $H$ solving the equation $F^{\flat}=F_1\circ H^{\flat}$; for instance, $H^{\flat}=g\circ F^{\flat}$ for $g\in Q(F_1)$.
\end{remark}

\begin{theorem}\label{thm:FactoriAWRI-BPA237111}
Assume AC and let $F\colon X^*\to Y$ be a function. Consider the following assertions.
\begin{enumerate}
\item[(i)] $F$ is preassociative and unarily quasi-range-idempotent.

\item[(ii)] There exists an associative $\varepsilon$-standard operation $H\colon X^*\to X\cup\{\varepsilon\}$ and a one-to-one function $f\colon\ran(H^{\flat})\to Y$ such that $F^{\flat}=f\circ H^{\flat}$.
\end{enumerate}
Then $(i) \Rightarrow (ii)$. If $F$ is standard, then $(ii) \Rightarrow (i)$. Moreover, if condition (ii) holds, then we have $F^{\flat}=F_1\circ H^{\flat}$, $f=F_1|_{\ran(H^{\flat})}$, $f^{-1}\in Q(F_1)$, and we may choose $H^{\flat}=g\circ F^{\flat}$ for any $g\in Q(F_1)$.
\end{theorem}

\begin{proof}
$(i)\Rightarrow (ii)$ Let $g\in Q(F_1)$ and let $H\colon X^*\to X\cup\{\varepsilon\}$ be the $\varepsilon$-standard operation defined by $H^{\flat}=g\circ F^{\flat}$. By Proposition~\ref{prop:QRIqi2431}, $H$ is unarily range-idempotent and we have $F^{\flat}=f\circ H^{\flat}$, where $f=F_1|_{\ran(H^{\flat})}$ is one-to-one. By Lemma~\ref{lemma:UQRIrew67}, $H$ is associative.

$(ii)\Rightarrow (i)$ $F$ is unarily quasi-range-idempotent by Proposition~\ref{prop:ACqriTFAE3451}. It is also preassociative by Proposition~\ref{prop:leftcomp56}.

To prove the second part of the result, we observe that $F_1\circ H^{\flat}=f\circ H_1\circ H^{\flat}=f\circ H^{\flat}$ and hence $F|_{\ran(H^{\flat})}=f|_{\ran(H^{\flat})}=f$. We then conclude by using Proposition~\ref{prop:ACqriTFAE3451}(iv).
\end{proof}

\begin{remark}\label{rem:fsd68s}
\begin{enumerate}
\item[(a)] If condition (ii) of Theorem~\ref{thm:FactoriAWRI-BPA237111} holds, then by Eq.~(\ref{eq:saf76sf5xx}) we see that $F$ can be computed recursively by
    $$
    F_n(x_1\cdots x_n) ~=~ F_2((g\circ F_{n-1})(x_1\cdots x_{n-1})x_n),\qquad n\geqslant 3,
    $$
    where $g\in Q(F_1)$. A similar observation was already made in a more particular setting for the so-called quasi-associative functions; see \cite{Yag87}.

\item[(b)] A standard function $F\colon X^*\to Y$ satisfying $F^{\flat}=F_1\circ H^{\flat}$, where $H$ is an associative $\varepsilon$-standard operation, need not be preassociative. The example given in Remark~\ref{rem:s8d76}(b) illustrates this observation. To give a second example, take $X=\R$, $F_1(x)=|x|$ (the absolute value of $x$) and $H_n(\bfx)=\sum_{i=1}^nx_i$ for every $n\in\N$. Then $F(1)=F(-1)$ but $F(11)=2\neq 0=F(1(-1))$. Thus $F$ is not preassociative.
\end{enumerate}
\end{remark}

We now provide necessary and sufficient conditions on the unary and binary parts for a standard function $F\colon X^*\to Y$ to be preassociative and unarily quasi-range-idempotent. The result is stated in Theorem~\ref{thm:sdf87fs2} below and follows from the next proposition.

\begin{proposition}\label{thm:sad6f7asdf}
Assume AC. A standard function $F\colon X^*\to Y$ is preassociative and unarily quasi-range-idempotent if and only if $\ran(F_2)\subseteq\ran(F_1)$ and there exists $g\in Q(F_1)$ such that
\begin{enumerate}
\item[(i)] $H_2=H_2\circ (H_1,\id)=H_2\circ (\id,H_1)$,

\item[(ii)] $H_2$ is associative, and

\item[(iii)] the following holds
$$
F_n(x_1\cdots x_n) ~=~ F_2((g\circ F_{n-1})(x_1\cdots x_{n-1})x_n),\qquad n\geqslant 3,
$$
or equivalently,
$$
F_n(x_1\cdots x_n) ~=~ F_2(H_2(\cdots H_2(H_2(x_1x_2)x_3)\cdots) x_n),\qquad n\geqslant 3,
$$
\end{enumerate}
where $H_1=g\circ F_1$ and $H_2=g\circ F_2$.
\end{proposition}

\begin{proof}
(Necessity) Let $F\colon X^*\to Y$ be standard, preassociative, and unarily quasi-range-idem{\-}potent. Then clearly $\ran(F_2)\subseteq\ran(F^{\flat})=\ran(F_1)$. Let $g\in Q(F_1)$ and let $H\colon X^*\to X\cup\{\varepsilon\}$ be the $\varepsilon$-standard operation defined by $H^{\flat}=g\circ F^{\flat}$. By Lemma~\ref{lemma:UQRIrew67}, $H$ is associative and hence conditions (i)--(iii) hold by Proposition~\ref{thm:Acz12new}.

(Sufficiency) Let $F\colon X^*\to Y$ be a standard function satisfying $\ran(F_2)\subseteq\ran(F_1)$ and conditions (i)--(iii) for some $g\in Q(F_1)$. By conditions (ii) and (iii) we must have
$$
F_n ~=~ F_2\circ (g\circ F_{n-1},\id) ~=~ F_2\circ (\id,g\circ F_{n-1}) \qquad\mbox{for every $n\geqslant 3$.}
$$
Then $\ran(F_n)\subseteq\ran(F_2)\subseteq\ran(F_1)$ for every $n\geqslant 3$ and hence $F$ is unarily quasi-range-idempotent.

Let us show that $F$ is preassociative. By Lemma~\ref{lemma:UQRIrew67} it suffices to show that the $\varepsilon$-standard operation $H\colon X^*\to X\cup\{\varepsilon\}$ defined by $H^{\flat}=g\circ F^{\flat}$ is associative. By Proposition~\ref{thm:Acz12new}, it suffices to show that $H_1=H_1\circ H_1$ and $H_2=H_1\circ H_2$, or equivalently, $g\circ F_1=g\circ F_1\circ g\circ F_1$ and $g\circ F_2=g\circ F_1\circ g\circ F_2$, respectively. These identities clearly hold since $g\in Q(F_1)$ implies $g\circ F_1\circ g=g$.
\end{proof}

\begin{theorem}\label{thm:sdf87fs2}
Assume AC and let $F_1\colon X\to Y$ and $F_2\colon X^2\to Y$ be two functions. Then there exists a preassociative and unarily quasi-range-idempotent standard function $G\colon X^*\to Y$ such that $G_1=F_1$ and $G_2=F_2$ if and only if $\ran(F_2)\subseteq\ran(F_1)$ and there exists $g\in Q(F_1)$ such that conditions (i) and (ii) of Proposition~\ref{thm:sad6f7asdf} hold, where $H_1=g\circ F_1$ and $H_2=g\circ F_2$. Such a function $G$ is then uniquely determined by $G_0$ and $G_n=G_2\circ (g\circ G_{n-1},\id)$ for $n\geqslant 3$.
\end{theorem}

We close this section by introducing and discussing a stronger version of preassociativity. Recall that preassociativity means that the equality of the function values of two strings still holds when adding identical arguments on the left or on the right of these strings. Now, suppose that the equality still holds when adding identical arguments at any place. We then have the following definition.

\begin{definition}\label{def:sdfssf}
We say that a function $F\colon X^*\to Y$ is \emph{strongly preassociative} if for every $\bfx\bfx'\bfy\bfz\bfz'\in X^*$ we have
\begin{equation}\label{eq:SrPA}
F(\bfx\bfz) ~=~ F(\bfx'\bfz')\quad\Rightarrow\quad F(\bfx\bfy\bfz) ~=~ F(\bfx'\bfy\bfz').
\end{equation}
\end{definition}

Clearly, Definition~\ref{def:sdfssf} remains unchanged if we assume that $|\bfy|=1$.

As we could expect, strongly preassociative functions are exactly those preassociative functions which are symmetric, i.e., invariant under any permutation of the arguments.

\begin{proposition}\label{prop:sdf57dzzs}
A function $F\colon X^*\to Y$ is strongly preassociative if and only if it is preassociative and $F_n$ is symmetric for every $n\in\N$.
\end{proposition}

\begin{proof}
We only need to prove the necessity. Taking $\bfz=\bfz'=\varepsilon$ or $\bfx=\bfx'=\varepsilon$ in Eq.~(\ref{eq:SrPA}) shows that $F$ is preassociative. Taking $\bfz=\bfx'=\varepsilon$ and $\bfz'=\bfx$, we obtain $F(\bfx\bfy)=F(\bfy\bfx)$ for every $\bfx\bfy\in X^*$. Then, by strong preassociativity we also have $F(\bfu\bfx\bfv\bfy\bfw)=F(\bfu\bfy\bfv\bfx\bfw)$ for every $\bfu\bfx\bfv\bfy\bfw\in X^*$, which shows that $F$ is symmetric.
\end{proof}

%---------------------------------------------------------------------------------------------- Section
\section{Concluding remarks and open problems}

We have proposed a relaxation of associativity for variadic functions, namely preassociativity, and we have investigated this new property. In particular, we have presented characterizations of those preassociative standard functions which are unarily quasi-range-idempotent.

This area of investigation is intriguing and appears not to have been previously studied. We have just skimmed the surface, and there are a lot of questions to be answered. Some are listed below.

\begin{enumerate}
\item Find necessary and sufficient conditions on a class of preassociative functions for each element of this class to be completely determined inside the class by its nullary, unary, and binary parts (cf.\ Remark~\ref{rem:s8d76xd}).

\item Find a generalization of Theorem~\ref{thm:FactoriAWRI-BPA237111} without the unary quasi-range-idem{\-}potence property.

\item Find necessary and sufficient conditions on $F_1$ for a standard function $F$ of the form $F^{\flat}=F_1\circ H^{\flat}$, where $H$ is an associative $\varepsilon$-standard operation, to be preassociative (cf.\ Remark~\ref{rem:fsd68s}(b)).
\end{enumerate}

%---------------------------------------------------------------------------------------------- Bibliography

\end{document}